\documentclass[12pt]{amsart}

\usepackage{amsmath, amssymb, amscd, newlfont}
\newcommand{\bbC}{{\mathbb{C}}}

\numberwithin{equation}{section}

\newtheorem{Lem}[equation]{Lemma}
\newtheorem{Def}[equation]{Definition}
\newtheorem{Thm}[equation] {Theorem}

\title
[
Higher Order Weierstrass Points
]
{%On $m$--fold Holomorphic Differentials and Modular Forms
  On Higher Order Weierstrass Points on $X_0(N)$
} 

\author{Damir Miko\v c and Goran Mui\'c }
\address{Department of Teacher Education Studies in Gospi\' c,
  University of Zadar, dr. Ante Star\v cevi\' ca 12,
53000 Gospi\' c, Croatia}
\email{damir.mikoc@gmail.com}
\address{
Department of Mathematics, Faculty of Science, 
University of Zagreb,
Bijeni\v cka 30, 10000 Zagreb,
Croatia}
 \email{gmuic@math.hr}

\begin{document}

\begin{abstract}
  Let $\Gamma$ be the Fuchsian group of the first kind. For an even integer $m\ge 4$,  we describe the space $H^{m/2}\left(\mathfrak R_\Gamma\right)$ of  $m/2$--holomorphic differentials in terms of a subspace
  $S_m^H(\Gamma)$  of the space of   (holomorphic) cuspidal modular forms $S_m(\Gamma)$. This generalizes classical isomorphism $S_2(\Gamma)\simeq H^{1}\left(\mathfrak R_\Gamma\right)$. We study the properties
  of $S_m^H(\Gamma)$. As an application,  we describe the algorithm implemented in SAGE for testing if a cusp at $\infty$ for non--hyperelliptic  $X_0(N)$ is a $\frac{m}{2}$--Weierstrass point.
  \end{abstract}
\subjclass[2000]{11F11}
\keywords{Wronskians, holomorphic differentials, cuspidal modular forms}
\thanks{The  author acknowledges Croatian Science Foundation grant IP-2018-01-3628.}
\maketitle

\section{Introduction}

Let $\Gamma$ be the Fuchsian group of the first kind  \cite[Section 1.7, page 28]{Miyake}. Examples of such groups are the important Hecke congruence groups   $\Gamma_0(N)$, $N\ge 1$.
Let $\Bbb H$ be the complex upper half-plane.  The quotient $\Gamma\backslash \Bbb H$ can be compactified by adding a finite number of $\Gamma$-orbits of points in $\mathbb R\cup \{\infty\}$ called cusps
of $\Gamma$ and we obtain a compact Riemann surface which will be denoted by $\mathfrak{R}_\Gamma$. For $\Gamma=\Gamma_0(N)$, we let $X_0(N)= \mathfrak{R}_\Gamma$.
For $x\in \mathbb H$ or  $x\in \mathbb R\cup \{\infty\}$ a cusp for $\Gamma$, let $\mathfrak a_x$ be the $\Gamma$--orbit of $x$ i.e.,
the corresponding point in  $\mathfrak{R}_\Gamma$. For $l\ge 1$, let $H^l\left(\mathfrak R_\Gamma\right)$ be the space of all holomorphic differentials on $\mathfrak{R}_\Gamma$ (see \cite{FK}, or
Section \ref{mhd}) in this paper).

Let $m\ge 2$ be an even integer.  Let
 $S_m(\Gamma)$ be the space of (holomorphic) cusp forms of weight $m$. It is well--known that
$S_2(\Gamma)$ is naturally isomorphic to the vector space $H^1\left(\mathfrak R_\Gamma\right)$ (see \cite[Theorem 2.3.2]{Miyake}). This is employed on many instances in studying various properties of
modular curves (see for example \cite[Chapter 6]{ono}).  In this paper we study the generalization of this concept to the holomorphic differentials of higher order.

For an even integer $m\ge 4$, the space $S_m(\Gamma)$ is usually too big to be isomorphic to $H^{m/2}\left(\mathfrak R_\Gamma\right)$ due to presence of cusps and elliptic
points. So, in general, $H^{m/2}\left(\mathfrak R_\Gamma\right)$ corresponds to a subspace $S^H_m(\Gamma)$ of $S_m(\Gamma)$ (see Lemma \ref{mhd-8000}). 
We study the space  $S^H_m(\Gamma)$ in detail in Section \ref{mhd-cont} (see Lemma \ref{mhd-9}).

We recall in Section \ref{mhd} (see Definition \ref{mhd-def}) that  $\mathfrak a \in \mathfrak R_\Gamma$ is a $m/2$-Weierstrass point if
there exists a non--zero $\omega\in H^{m/2}\left(\mathfrak R_\Gamma\right)$ such that
$\nu_{\mathfrak a}(\omega)\ge \dim H^{m/2}\left(\mathfrak R_\Gamma\right)$. 
When $m=2$ we speak about classical Weierstrass points. So, $1$-Weierstrass points are simply Weierstrass points. Weierstrass points on modular curves are very-well studied
(see for example \cite[Chapter 6]{ono}, \cite{neeman}, \cite{Ogg}, \cite{pete-1}, \cite{pete-2}, \cite{roh}, \cite{rd}). Higher--order Weierstrass points has not been
not studied much (see for example \cite{neeman}, \cite{olsen}).

\vskip .2in
The case $m\ge 4$ is more complex. It is studied in Section \ref{algo}. We recall that $\mathfrak R_\Gamma$ is hyperelliptic if $g(\Gamma)\ge 2$, and there is a degree two map onto $\mathbb P^1$.
Under the assumptions that  $\mathfrak R_\Gamma$ is not hyperelliptic and  that $\mathfrak a_\infty$ is a cusp for $\Gamma$, we develop a criterion that
$\mathfrak a_\infty$ is  a   $\frac{m}{2}$--Weierstrass point for $\mathfrak R_\Gamma$ (see Theorem \ref{cts-500000}). The Section \ref{algo} is mainly  devoted to the proof of
Theorem \ref{cts-500000}. For $\Gamma=\Gamma_0(N)$, $X_0(N)$  is not hyperelliptic for most of values of $N$ (see below), and Theorem \ref{cts-500000} contains an algorithm for testing that
$\mathfrak a_\infty$ is  a   $\frac{m}{2}$--Weierstrass point. We illustrate this by several examples (see the end of  Section \ref{algo}). 

Let $\Gamma=\Gamma_0(N)$, $N\ge 1$. We recall that  $g(\Gamma_0(N))\ge 2$ unless
$$
\begin{cases}
N\in\{1-10, 12, 13, 16, 18, 25\} \ \ \text{when $g(\Gamma_0(N))=0$, and}\\
N\in\{11, 14, 15, 17, 19-21, 24, 27, 32, 36, 49\} \ \ \text{when
	$g(\Gamma_0(N))=1$.}
\end{cases}
$$
Let $g(\Gamma_0(N))\ge 2$. Then,  we remark that Ogg \cite{Ogg} has determined all $X_0(N)$ which are hyperelliptic curves.
In view of Ogg's paper, we see that $X_0(N)$ is {\bf not hyperelliptic} for
$N\in \{34, 38, 42, 43, 44, 45, 51-58, 60-70\}$ or $N\ge 72$. This implies $g(\Gamma_0(N))\ge 3$.

In this paper we continue our earlier approach in studying various aspects of modular curves (\cite{Kodrnja1}, \cite{Muic1}, \cite{Muic2}, \cite{MuicKodrnja}, \cite{MuicMi}).
This paper contains large parts of previous manuscript   \cite{MuicMi1}. The rest of manuscript  \cite{MuicMi1}, related to cups forms constructed out of Wronskians, would be extended and
published separately since does not fit here. 

We would  like to thank I. Kodrnja for her help with the SAGE system. Also, we would like to thank M. Kazalicki and F. Najman for some useful discussions about modular
forms and curves in general.

\section{Preliminaries I: Holomorphic Differentials and $m$--Weierstrass Points}\label{mhd}
The goal of the present section is to recall necessary facts about holomorphic differentials and $m$--Weierstrass Points on a general compact Riemann surface, phrased in terms of
$\mathfrak R_\Gamma$ where  $\Gamma$ is  a Fuchsian group of the first kind. We let $D^m\left(\mathfrak R_\Gamma\right)$ (resp., $H^m\left(\mathfrak R_\Gamma\right)$)be the space of meromorphic (resp., holomorphic)
differential of degree $m$ on $\mathfrak R_\Gamma$ for each $m\in \mathbb Z$. We recall that $D^0\left(\mathfrak R_\Gamma\right)=\mathbb C\left(\mathfrak R_\Gamma\right)$, and
$D^m\left(\mathfrak R_\Gamma\right)\neq 0$ for all other  $m\in \mathbb Z$. In fact, if we fix a non--zero $\omega \in  D^1\left(\mathfrak R_\Gamma\right)$, then
$D^m\left(\mathfrak R_\Gamma\right)=\mathbb C\left(\mathfrak R_\Gamma\right)\omega^n$. We have  the following:
\begin{equation}\label{mhd-1}
  \deg{\left(\mathrm{div}{(\omega)}\right)}= 2m(g(\Gamma)-1), \ \ \omega \in D^m\left(\mathfrak R_\Gamma\right),  \ \omega\neq 0.
\end{equation}

We shall be interested in the case $m\ge 1$, and in holomorphic differentials. We recall \cite[Proposition III.5.2]{FK} that 
\begin{equation}\label{mhd-2}
  \dim H^m\left(\mathfrak R_\Gamma\right)=
  \begin{cases}  0 \ \ &\text{if} \ \ m\ge 1, g(\Gamma)=0;\\
     g(\Gamma)    \ \ &\text{if} \ \ m=1, \ g(\Gamma)\ge 1;\\
   g(\Gamma)    \ \ & \text{if} \ \ m\ge 2, \ g(\Gamma)= 1;\\
     (2m-1)\left(g\left(\mathfrak R_\Gamma\right)-1\right)    \ \ &\text{if} \ \ m\ge 2, \ g(\Gamma)\ge 2.\\           
  \end{cases}
  \end{equation}

This follows easily from Riemann-Roch theorem. Recall that a canonical class $K$ is simply a divisor on any non--zero meromorphic form $\omega$ on $\mathfrak R_\Gamma$. Different choices of a
$\omega$ differ by a divisor of a non--zero function $f\in \mathbb C\left(\mathfrak R_\Gamma\right)$
$$
\mathrm{div}{(f\omega)}=\mathrm{div}{(f)}+ \mathrm{div}{(\omega)}.
$$
Different choices of $\omega$ have the same degree since $\deg{\left(\mathrm{div}{(f)}\right)}=0$.

For a divisor $\mathfrak a$, we let
$$
L(\mathfrak a)=\left\{f\in \mathbb C\left(\mathfrak R_\Gamma\right); \ \ f=0 \ \text{or} \ \mathrm{div}{(f)}+ \mathfrak a\ge 0\right\}.
$$
We have the following three facts:
\begin{itemize}
\item[(1)] for $\mathfrak a=0$, we have $L(\mathfrak a)=\mathbb C$;
\item[(2)] if $\deg{(\mathfrak a)}< 0$, then  $L(\mathfrak a)=0$;
  \item[(3)] the Riemann-Roch theorem: $\dim L(\mathfrak a)= \deg{(\mathfrak a)}-g(\Gamma)+1 + \dim L(K-\mathfrak a)$.
  \end{itemize}

Now, it is obvious that   $f\omega^m \in H^m\left(\mathfrak R_\Gamma\right)$ if and only if
$$
\mathrm{div}{(f\omega^m)}=\mathrm{div}{(f)}+m\mathrm{div}{(\omega)}= \mathrm{div}{(f)}+mK\ge 0.
$$
Equivalently, $f\in L(mK)$. Thus, we have that $\dim H^m\left(\mathfrak R_\Gamma\right)=  \dim L(mK)$. Finally, by the Riemann-Roch theorem, we have the following:
$$
\dim L(mK)=\deg{(mK)}-g(\Gamma)+1 + \dim L((1-m)K)=(2m-1)(g\left(\mathfrak R_\Gamma\right)-1)+ \dim L((1-m)K).
$$
Now, if $g(\Gamma)\ge 2 $, then  $\deg{(K)}=2(g(\Gamma)-1)>0$, and the claim easily follows from (1) and (2) above.
Next, assume that $g(\Gamma)=1$. If $\omega\in \dim H^1\left(\mathfrak R_\Gamma\right)$ s non--zero, then it has a degree zero.
Thus, it has no zeroes. This means that $\omega H^{l-1}\left(\mathfrak R_\Gamma\right)=H^l\left(\mathfrak R_\Gamma\right)$ for all $l\in \mathbb Z$. But since obviously $H^0\left(\mathfrak R_\Gamma\right)$ consists
of constants only, we obtain the claim. Finally, the case $g(\Gamma)=0$ is obvious from (2) since the degree of $mK$ is $2m(g(\Gamma)-1)<0$ for all $m\ge 1$.

\vskip .2in

Assume that  $g(\Gamma)\ge 1$ and $m\ge 1$. Then,  $\dim H^m\left(\mathfrak R_\Gamma\right)\neq 0$. Let $t= \dim H^m\left(\mathfrak R_\Gamma\right)$. We fix the basis
$\omega_1, \ldots, \omega_t$ of  $H^m\left(\mathfrak R_\Gamma\right)$. Let $z$ be any local coordinate on $\mathfrak R_\Gamma$. Then, locally there exists unique holomorphic functions
$\varphi_1, \ldots, \varphi_t$ such that $\omega_i=\varphi_i \left(dz\right)^m$, for all $i$. Then, again locally,
we can consider the Wronskian $W_z$ defined by
\begin{equation}\label{mhd-3}
W_z\left(\omega_1, \ldots, \omega_t\right)\overset{def}{=}\left|\begin{matrix}
\varphi_1(z) &
\cdots & \varphi_{t}(z) \\
\frac{d\varphi_1(z)}{dz} &
\cdots & \frac{d\varphi_{t}(z)}{dz} \\
&\cdots & \\
\frac{d^{t-1}\varphi_1(z)}{dz^{k-1}} &
\cdots & \frac{d^{t-1}\varphi_{t}(z)}{dz^{t-1}} \\
\end{matrix}\right|. 
\end{equation}

As proved in \cite[Proposition III.5.10]{FK}, collection of all
\begin{equation}\label{mhd-4}
W_z\left(\omega_1, \ldots, \omega_t\right)\left(dz\right)^{\frac{t}{2}\left(2m-1+t\right)} ,
\end{equation}
defines a non--zero holomorphic differential form
$$
W\left(\omega_1, \ldots, \omega_t\right)\in H^{\frac{t}{2}\left(2m-1+t\right)}\left(\mathfrak R_\Gamma\right).
$$
We call this form the Wronskian of the basis $\omega_1, \ldots, \omega_t$. It is obvious that a different choice of a basis of
$H^m\left(\mathfrak R_\Gamma\right)$ results in a Wronskian which differ from $W\left(\omega_1, \ldots, \omega_t\right)$ by a multiplication by
a non--zero complex number. Also,  the degree is given by
\begin{equation}\label{mhd-5a}
\deg{\left(\mathrm{div}{\left(W\left(\omega_1, \ldots, \omega_t\right)\right)}\right)}= t\left(2m-1+t\right)(g\left(\mathfrak R_\Gamma\right)-1).
\end{equation}

\vskip .2in
Following \cite[III.5.9]{FK}, we make the following definition:

\begin{Def}\label{mhd-def}  Let $m\ge 1$ be an integer. We say that $\mathfrak a \in \mathfrak R_\Gamma$ is a $m$-Weierstrass point if
there exists a non--zero $\omega\in H^m\left(\mathfrak R_\Gamma\right)$ such that
$$
\nu_{\mathfrak a}(\omega)\ge \dim H^m\left(\mathfrak R_\Gamma\right).
$$
\end{Def}
Equivalently \cite[Proposition III.5.10]{FK} , if 
$$
\nu_{\mathfrak a}\left(W\left(\omega_1, \ldots, \omega_t\right)\right)\ge 1.
$$
When $m=1$ we speak about classical Weierstrass points. So, $1$-Weierstrass points are simply Weierstrass points.

\section{Preliminaries II: Interpretation in Terms of Modular Forms} \label{mhd-cont}

In this section we give interpretation of results of Section \ref{mhd} in terms of modular forms.
Again, $\Gamma$ stand for a Fuschsian group of the first kind. Let $m\ge 2$ be an even integer. We consider the space
$\mathcal A_m(\Gamma)$ be the space of all all meromorphic functions $f: \mathbb H\rightarrow \bbC$ 
such that $f(\gamma.z)=j(\gamma, z)^m f(z)$  ($z\in \mathbb H$, $\gamma\in \Gamma$) which are meromorphic at
every cusp for $\Gamma$. By \cite[Theorem 2.3.1]{Miyake}, there exists isomorphism of vector spaces $\mathcal A_m(\Gamma)\longrightarrow D^{m/2}\left(\mathfrak R_\Gamma\right)$,
denoted by $f\longmapsto \omega_f$ such that the following holds (see \cite[Theorem 2.3.3]{Miyake}):

\begin{equation}\label{mhd-5}
  \begin{aligned}
  &\nu_{\mathfrak a_\xi}(f)=\nu_{\mathfrak a_\xi}(\omega_f) +\frac{m}{2}\left(1- \frac{1}{e_{\mathfrak a_\xi}}\right) \ \ \text{if} \ \xi\in \mathbb H\\
    &\nu_{\mathfrak a}(f)=\nu_{\mathfrak a}(\omega_f) + \frac{m}{2} \ \ \text{for $\Gamma$---cusp $\mathfrak a$.}\\
    & \mathrm{div}{(f)}=\mathrm{div}{(\omega_f)}+\sum_{\mathfrak a \in \mathfrak R_\Gamma} \frac{m}{2}\left(1- \frac{1}{e_{\mathfrak a}}\right) \mathfrak a,
     \end{aligned}
\end{equation}
where $1/e_{\mathfrak a}=0$ if $\mathfrak a$ is a cusp, and $\mathfrak a_\xi\in \mathfrak R_\Gamma$ is the projection of $\xi$
to $\mathfrak R_\Gamma$.  Let $f\in M_m(\Gamma)$ (a space of modular forms). Using (\ref{mhd-5}), we obtain

\begin{equation}\label{mhd-6}
\mathrm{div}{(\omega_f)}=\mathfrak c' _f -\sum_{\mathfrak a\in \mathfrak
R_\Gamma, \ elliptic} \left[\frac{m}{2}(1-1/e_{\mathfrak  a})\right]\mathfrak a - \frac{m}{2}\sum_{\mathfrak b\in \mathfrak
R_\Gamma, \ cusp}  \mathfrak b,
\end{equation}
where as usual in our previous papers (see for example \cite[Lemma 2.2]{MuicKodrnja}), $\mathfrak c'_f$ is an effective integral divisor
which satisfies
\begin{equation}\label{mhd-60000000000000000}
\mathrm{div}{(f)}=  \mathfrak c'_f+ \sum_{\mathfrak a\in \mathfrak
R_\Gamma, \ \ elliptic} \left(\frac{m}{2}(1-1/e_{\mathfrak a}) -
\left[\frac{m}{2}(1-1/e_{\mathfrak
    a})\right]\right)\mathfrak a.
\end{equation}
This shows that $\omega_f$ is holomorphic everywhere except maybe at cusps and elliptic points.
Moreover, if $f\in  S_m(\Gamma)$ (a space of cuspidal modular forms), then 
\begin{equation}\label{mhd-7}
\mathrm{div}{(\omega_f)}=\mathfrak c_f -  \sum_{\mathfrak a\in \mathfrak
R_\Gamma, \ elliptic} \left[\frac{m}{2}(1-1/e_{\mathfrak  a})\right]\mathfrak a - \left(\frac{m}{2} -1\right)\sum_{\mathfrak b\in \mathfrak
R_\Gamma, \ cusp}  \mathfrak b,
\end{equation}

Next, we determine all $f\in M_m(\Gamma)$ such that $\omega_f\in H^{m/2}\left(\mathfrak R_\Gamma\right)$. From (\ref{mhd-6}) we see that
such $f$ must belong to $S_m(\Gamma)$, and from (\ref{mhd-7})
\begin{equation}\label{mhd-8}
\mathfrak c_f\ge \sum_{\mathfrak a\in \mathfrak
R_\Gamma, \ elliptic} \left[\frac{m}{2}(1-1/e_{\mathfrak  a})\right]\mathfrak a+ \left(\frac{m}{2} -1\right)\sum_{\mathfrak b\in \mathfrak
  R_\Gamma, \ cusp}  \mathfrak b,
\end{equation}
 where the  integral divisor $c_f$ is defined by (see \cite[Lemma 2.2]{MuicKodrnja})
$$ \mathfrak c_f\overset{def}{=}\mathfrak c'_f-
\sum_{\substack{\mathfrak b \in \mathfrak
    R_\Gamma, \\ cusp}} \mathfrak b.
$$

Now, above considerations immediately imply the following result:
\begin{Lem}\label{mhd-8000}
We define the subspace of $S_m(\Gamma)$ by
  $$
  S^H_m(\Gamma)=\left\{f\in S_m(\Gamma); \ \  \text{$f=0$ or $f$ satisfies (\ref{mhd-8})} \right\}.
  $$
  Then, the  map  $f\longmapsto \omega_f$ is an isomorphism  of $S^H_m(\Gamma)$ onto $H^{m/2}\left(\mathfrak R_\Gamma\right)$.
\end{Lem}  
  
  We remark that when $m=2$,  (\ref{mhd-8}) and reduces to obvious $\mathfrak c_f\ge 0$. Hence,
  $S^H_2(\Gamma)= S_2(\Gamma)$ recovering  the standard
  isomorphism of $S_2(\Gamma)$ and $H^1(\mathfrak R_\Gamma)$ (see \cite[Theorem 2.3.2]{Miyake}). We continue by collecting a few properties 
  of spaces $S^H_m(\Gamma)$. 

  \begin{Lem} \label{mhd-9} Assume that $m, n\ge 2$  are even integers. Let $\Gamma$ be a Fuchsian group of the first kind. Then, we have  the following:
    \begin{itemize}
    \item[(i)] $S^H_2(\Gamma)= S_2(\Gamma)$.
    \item[(ii)]  $S^H_m(\Gamma)$ is isomorphic to  $H^{m/2}\left(\mathfrak R_\Gamma\right)$.
    \item[(iii)]  $S^H_m(\Gamma)=\{0\}$ if $g(\Gamma)=0$.
    \item[(iv)] Assume that $g(\Gamma)=1$.  Let us write $S_2(\Gamma)=\mathbb C \cdot f$, for some non--zero cuspidal form $f$. Then, we have
      $S^H_m(\Gamma)= \mathbb C \cdot f^{m/2}$. 
    \item[(v)] $\dim{S^H_m(\Gamma)}= (m-1)\left(g(\Gamma)-1\right)$ if $g(\Gamma)\ge 2$.
    \item[(vi)] $S^H_m(\Gamma)\cdot S^H_n(\Gamma)\subset S^H_{m+n}(\Gamma)$.
       \item[(vii)] There are no  $m/2$--Weierstrass points on $\mathfrak R_\Gamma$ for $g(\Gamma)\in \{0, 1\}$. 
    \item[(viii)] Assume that $g(\Gamma)\ge 2$, and $\mathfrak a_\infty$ is a $\Gamma$-cusp.  Then,  $\mathfrak a_\infty$ is a 
      $\frac{m}{2}$--Weierstrass point if and only if there exists $f\in  S^H_m(\Gamma)$, $f\neq 0$, such that
      $$
      \mathfrak c'_f(\mathfrak a_\infty)\ge \begin{cases}   \frac{m}{2} + g(\Gamma) \ \ \text{if} \ \ m=2;\\
        \frac{m}{2} +  (m-1)(g(\Gamma)-1)  \ \ \text{if} \ \ m\ge 4.\\
        \end{cases}
      $$
    \item[(ix)] Assume that $g(\Gamma)\ge 1$, and $\mathfrak a_\infty$ is a $\Gamma$-cusp.  Then, there exists a basis $f_1, \ldots f_t$ of $S_m^H(\Gamma)$ such that
      their $q$--expansions are of the form
      $$
      f_u=a_uq^{i_u}+ \text{higher order terms in $q$}, \ \ 1\le u\le t,
      $$
      where 
      $$
      \frac{m}{2}\le i_1< i_2< \cdots < i_t \le \frac{m}{2}+ m\left(g(\Gamma)-1\right),
      $$
      and
      $$
      a_u\in \mathbb C, \ \  a_u\neq 0.
      $$
      \item[(x)] Assume that $g(\Gamma)\ge 1$, and $\mathfrak a_\infty$ is a $\Gamma$-cusp. Then,  $\mathfrak a_\infty$ is {\bf not} a 
      $\frac{m}{2}$--Weierstrass point if and only if there exists a basis $f_1, \ldots f_t$ of $S_m^H(\Gamma)$ such that
      their $q$--expansions are of the form
      $$
      f_u=a_uq^{u+m/2-1}+ \text{higher order terms in $q$}, \ \ 1\le u\le t,
      $$
      where 
      $$
      a_u\in \mathbb C, \ \  a_u\neq 0.
      $$
      \item[(xi)] Assume that $g(\Gamma)\ge 1$. 
    Let us fix a basis $f_1, \ldots, f_t$ of $S^H_m(\Gamma)$, and let $\omega_1, \ldots, \omega_t$ be the corresponding basis of $H^{m/2}\left(\mathfrak R_\Gamma\right)$. As in Section \ref{mhd},
    we construct
    holomorphic differential $W\left(\omega_1, \ldots, \omega_t\right)\in H^{\frac{t}{2}\left(m-1+t\right)}\left(\mathfrak R_\Gamma\right)$. We also construct the Wronskian $W(f_1,\ldots, f_t)\in
      S_{t(m+t-1)}(\Gamma)$  (see Lemma \ref{wron-2} below). Then, we have the following equality $\omega_{W(f_1,\ldots, f_t)}=W\left(\omega_1, \ldots, \omega_t\right)$.
    In particular, we obtain the following:
    $W(f_1,\ldots, f_t)\in S^H_{t(m+t-1)}(\Gamma)$. 
    Moreover, assume that $\mathfrak a_\infty$ is a $\Gamma$-cusp. Then, $\mathfrak a_\infty$ is a 
    $\frac{m}{2}$--Weierstrass point if and only if
    $$
    \mathfrak c_{ W(f_1,\ldots, f_t)}(\mathfrak a_\infty)\ge \frac{t}{2}\left(m-1+t\right).
    $$
      \end{itemize}  
    \end{Lem} 
  \begin{proof} (i) and (ii) follow from above discussion. Next, using above discussion and (\ref{mhd-2}) we obtain
    $$
    \dim{S^H_m(\Gamma)} =\dim H^{m/2}\left(\mathfrak R_\Gamma\right)=\begin{cases}  0 \ \ &\text{if} \ \ m\ge 2, g(\Gamma)=0;\\
     g(\Gamma)    \ \ &\text{if} \ \ m=2, \ g(\Gamma)\ge 1;\\
   g(\Gamma)    \ \ & \text{if} \ \ m\ge 4, \ g(\Gamma)= 1;\\
     (m-1)\left(g(\Gamma)-1\right)    \ \ &\text{if} \ \ m\ge 4, \ g(\Gamma)\ge 2.\\           
    \end{cases}
    $$
    This immediately implies (iii) and (v). Next, assume that $g(\Gamma)=1$. Then, we see that $\dim{S^H_m(\Gamma)}\le 1$ for all even integers $m\ge 4$.
    It is well known that $f^{m/2} \in S_m(\Gamma)$. Next, (\ref{mhd-8}) for $m=2$,  and \cite[Lemma 2.2]{MuicKodrnja} imply  $\mathrm{div}{(\omega_f)}=\mathfrak c_f=0$. 
    Using \cite[Theorem 2.3.2]{Miyake}, we obtain
    $\omega_{f^{m/2}}=\omega^{m/2}_f$. Hence,
    $$
    \mathrm{div}{(\omega_{f^{m/2}})}=\frac{m}{2} \mathrm{div}{(\omega_f)}=0.
    $$
    Then, applying (\ref{mhd-7}) with $f^{m/2}$ in place of $f$, we obtain
    $$
    \mathfrak c_{f^{m/2}} = \sum_{\mathfrak a\in \mathfrak
R_\Gamma, \ elliptic} \left[\frac{m}{2}(1-1/e_{\mathfrak  a})\right]\mathfrak a + \left(\frac{m}{2} -1\right)\sum_{\mathfrak b\in \mathfrak
      R_\Gamma, \ cusp}  \mathfrak b.
    $$
    This shows that $f^{m/2}\in S^H_m(\Gamma)$ proving (iv). (vi) follows from \cite[Theorem 2.3.1]{Miyake}.
(vii) follows immediately form the discussion in Section \ref{mhd}, and it is well--known. (viii) is a reinterpretation of
    Definition \ref{mhd-def}. The details are left to the reader as an easy exercise. Now, we prove (ix) and (x). The case of $g(\Gamma)=1$ are obvious since 
    we have $S_2(\Gamma)= \mathbb C \cdot f$ where
$$
\mathfrak c'_f=\mathfrak a_\infty + \sum_{\substack{\mathfrak b \in \mathfrak
    R_\Gamma, \ cusp\\
    \mathfrak b \neq \mathfrak a_\infty}} \mathfrak b.
$$
Next, we prove  (ix) and (x) in the case $g(\Gamma)\ge 2$. Let $f\in S_m^H(\Gamma)$, $f\neq 0$. Then, by the definition of $S_m^H(\Gamma)$, we obtain
\begin{equation} \label{mhd-11}
  \mathfrak c'_f(\mathfrak a_\infty)=1+ \mathfrak c_f(\mathfrak a_\infty)\ge 1+ \left(\frac{m}{2}-1\right)=\frac{m}{2}.
  \end{equation}
On the other hand, again by the definition of $S_m^H(\Gamma)$ (see (\ref{mhd-8})) and the fact that $\mathfrak c'_f\ge 0$, we obtain

\begin{align*}
  \deg{(\mathfrak c'_f)} & =\sum_{\mathfrak a \in \mathfrak R_\Gamma} \ \mathfrak c'_f(\mathfrak a) \ge 
                        \sum_{\mathfrak a\in \mathfrak R_\Gamma, \ elliptic}  \mathfrak c'_f(\mathfrak a) +  \sum_{\substack{\mathfrak b \in \mathfrak
      R_\Gamma, \ cusp\\    \mathfrak b \neq \mathfrak a_\infty}} \mathfrak c'_f(\mathfrak b) + \mathfrak c'_f(\mathfrak a_\infty)\ge\\
   &\sum_{\mathfrak a\in \mathfrak R_\Gamma, \ elliptic} \left[\frac{m}{2}(1-1/e_{\mathfrak  a})\right]+ \frac{m}{2} \left(t-1\right) + \mathfrak c'_f(\mathfrak a_\infty)
  \end{align*}
where $t$ is the number of nonequivalent $\Gamma$--cusps. The degree $\deg{(\mathfrak c'_f)}$ is given by  

\begin{align*}
  \deg{(\mathfrak c'_f)} &=   \dim M_m(\Gamma)+ g(\Gamma)-1\\
  &=   \begin{cases} 2(g(\Gamma)-1)+t  \ \ \text{if} \ \ m=2;\\
                 m(g(\Gamma)-1)+ \frac{m}{2}t+ \sum_{\substack{\mathfrak a\in \mathfrak
        R_\Gamma, \\ elliptic}} \left[\frac{m}{2}(1-1/e_{\mathfrak a})\right] \ \ \text{if} \ \ m\ge 4.\\
    \end{cases}
\end{align*}
Combining with the previous inequality, we obtain
$$
\mathfrak c'_f(\mathfrak a_\infty) \le \frac{m}{2}+ m(g(\Gamma)-1) \ \ \text{if} \ \ m\ge 2.
$$
Having in mind (\ref{mhd-11}), the rest of (ix) has a standard argument (see for example \cite[Lemma 4.3]{Muic}).  Finally, (x) follows (viii) and (ix).
The last claim (xi) follows easily if we note that $ \omega_{W(f_1,\ldots, f_t)}=W\left(\omega_1, \ldots, \omega_t\right)$ is equality of two meromorphic differentials. So, it is enough to pick a non--elliptic
point $w\in \mathbb H$ and check the equality in a small neighborhood of $w$ in $\mathbb H$. But this local identity is obvious, and the claim (xi) follows. 
\end{proof}

  \vskip .2in
  We end this section by recalling  a generalization of the usual notion of the Wronskian of modular forms \cite{roh}, (\cite{ono}, 6.3.1),  (\cite{Muic}, the proof of Theorem 4-5),
and (\cite{Muic2}, Lemma 4-1). 

\begin{Lem}\label{wron-2} Let $m\ge 1$. Then, for any sequence 
$f_1, \ldots, f_k\in M_m(\Gamma)$,
the Wronskian
$$
W\left(f_1, \ldots, f_k\right)(z)\overset{def}{=}\left|\begin{matrix}
f_1(z) &
\cdots & f_{k}(z) \\
\frac{df_1(z)}{dz} &
\cdots & \frac{df_{k}(z)}{dz} \\
&\cdots & \\
\frac{d^{k-1}f_1(z)}{dz^{k-1}} &
\cdots & \frac{d^{k-1}f_{k}(z)}{dz^{k-1}} \\
\end{matrix}\right|
$$
is a cuspidal modular form in $S_{k(m+k-1)}(\Gamma)$ if $k\ge 2$. If $f_1, \ldots, f_k$ are linearly independent,
then  $W\left(f_1, \ldots, f_k\right)\neq 0$.   
\end{Lem}

\section{An Algorithm For $X_0(N)$ } \label{algo} 
In this section we describe the algorithm for testing $\mathfrak a_\infty$ to be a $\frac{m}{2}$--Weierstrass point on $X_0(N)$ for all even integers
$m=2, 4, 6, \ldots$ assuming that $X_0(N)$ is not hyperelliptic (see Introduction).

We begin with the following remark. The criterion in Lemma \ref{mhd-9} (x) is a quite good criterion to check whether or not $\mathfrak a_\infty$ is a Weierstrass points (the case $m=2$) using
computer systems such as SAGE since we need just to list the basis. This case is well-known (see \cite[Definition 6.1]{ono}). This criterion has been used in practical computations in
combination with SAGE in \cite{MuicKodrnja} for $\Gamma=\Gamma_0(N)$.

But,  Lemma \ref{mhd-9} (x)  is not good when $m\ge 4$ in most interesting cases.  For example, when  $\Gamma$ has elliptic
points and for  $m$ large enough,  there is a  basis  of the space of cuspidal forms $S_m(\Gamma)$ which contains properly normalized cusp forms having leading terms
$q^{m/2}, \ldots,  q^{m/2+ m(g(\Gamma)-1)}$. This follows from the following two lemmas. 

First, we recall  \cite[Lemma 2.9]{Muic2} which is well-known in a slightly different notation (\cite{pete-1}, \cite{pete-2}):
\begin{Lem}\label{mhd-12} Let $m\ge 4$ be an even integer such that $\dim S_m(\Gamma)\ge g(\Gamma)+1$.
Then, for all 
$1\le i \le t_m-g$, there exists $f_i\in S_m(\Gamma)$ such that 
$\mathfrak c'_{f_i}(\mathfrak a_\infty)= i$.
\end{Lem}
The second lemma is even more elementary, and it follows from the explicit formula for the dimension of $S_m(\Gamma)$. The details are left to the reader as an exercise.
\begin{Lem}\label{mhd-13} Assume that $\Gamma$ has elliptic points. (For example, $\Gamma=\Gamma_0(N)$.) Then, for a sufficiently large even integer  $m$, we have
  $\frac{m}{2} +  m(g(\Gamma)-1)  \le \dim S_m(\Gamma) -g(\Gamma)$.
\end{Lem}
Thus,  under above assumptions,  in view of  Lemma \ref{mhd-9} (ix), the listing of basis of $S_m(\Gamma_0(N))$ in SAGE does not give any information about basis of 
$S^H_m(H)$.

\vskip .2in
Now, explain the algorithm for testing that  $\mathfrak a_\infty$ is  a 
$\frac{m}{2}$--Weierstrass point for $m\ge 4$. In what follows we assume that $g(\Gamma)\ge 2$ (see Lemma \ref{mhd-9} (vii)). 
We start with the following lemma. 
 
\begin{Lem}\label{cts-5000}  Let $m\ge 4$ be an even integer. Let us select a basis $f_0, \ldots, f_{g-1}$, $g=g(\Gamma)$, of $S_2(\Gamma)$. Then,
  all of $\binom{g+\frac{m}{2}-1}{\frac{m}{2}}$ monomials  $f_0^{\alpha_0}f_1^{\alpha_1}\cdots f_{g-1}^{\alpha_{g-1}}$, $\alpha_i\in \mathbb Z_{\ge 0}$, $\sum_{i=0}^{g-1} \alpha_i=\frac{m}{2}$,
  belong to $S_m^H(\Gamma)$. We denote  by $S_{m, 2}^H(\Gamma)$ this subspace of $S_m^H(\Gamma)$. 
  \end{Lem}
\begin{proof} This follows from Lemma \ref{mhd-9} (vi) since $S_2(\Gamma)=S_2^H(\Gamma)$ (see Lemma \ref{mhd-9} (i)).
 \end{proof}

\vskip .2in
The next lemma is crucial for the algorithm.

 \begin{Lem}\label{cts-50000}  Let $m\ge 4$ be an even integer. Assume that  $\mathfrak R_\Gamma$ is not hyperelliptic. Then, we have
   $S_{m, 2}^H(\Gamma)= S_m^H(\Gamma)$.
 \end{Lem}
 \begin{proof} We use notation of Section \ref{mhd} freely. The reader should review Lemma \ref{mhd-9}.
    Let $F\in S_2(\Gamma)$, $F\neq 0$.  We define a holomorphic differential form $\omega \in H\left(\mathfrak R_\Gamma\right)$
    by $\omega=\omega_F$. Define a canonical class $K$ by $K=\mathrm{div}{(\omega)}$. We prove the following:
    \begin{equation}\label{cts-2}
  L\left(\frac{m}{2}K\right)=\left\{\frac{f}{F^{m/2}}; \ \ f\in S_m^H(\Gamma) \right\}.
    \end{equation}
The case $m=2$ is of course well--known.   By  the Riemann-Roch theorem and standard results recalled in Section \ref{mhd} we have
  \begin{align*}
    \dim L\left(\frac{m}{2}K\right) &= \deg{\left(\frac{m}{2}K\right)}-g(\Gamma)+1 + \dim L\left(\left(1- \frac{m}{2}\right)K\right)\\
    &=   (m-1)(g(\Gamma)-1)+\begin{cases} 1 \ \ \text{if} \ \ m=2;\\
                                          0 \ \ \text{if} \ \ m\ge 4.\\
                             \end{cases}           
     \end{align*}

Next,  we recall that $S_2(\Gamma)=S_2^H(\Gamma)$  (see Lemma \ref{mhd-9} (i)). Then, Lemma \ref{mhd-9} (vi) we obtain 
$F^{\frac{m}{2}}\in S_m^H(\Gamma)$. Therefore,  $f/F^{\frac{m}{2}}\in  \mathbb C\left(\mathfrak R_\Gamma\right)$ for all $f\in S_m^H(\Gamma)$.

By the correspondence described in (\ref{mhd-5}) we have

\begin{align*}
\mathrm{div}{(F)}=\mathrm{div}{(\omega_F)}+\sum_{\mathfrak a \in \mathfrak R_\Gamma} \left(1- \frac{1}{e_{\mathfrak a}}\right) \mathfrak a &=
K+ \sum_{\mathfrak a \in \mathfrak R_\Gamma} \left(1- \frac{1}{e_{\mathfrak a}}\right) \mathfrak a\\
&= K+  \sum_{\mathfrak a\in \mathfrak R_\Gamma, \ \ elliptic} (1-1/e_{\mathfrak a})  \mathfrak a + \sum_{\substack{\mathfrak b \in \mathfrak R_\Gamma, \\ cusp}} \mathfrak b.
\end{align*}

Thus, for  $f\in S_m^H(\Gamma)$, we have the following:
\begin{align*}
  \mathrm{div}{\left(\frac{f}{F^{\frac{m}{2}}}\right)} + \frac{m}{2}K &= \mathrm{div}{(f)}- \frac{m}{2}\mathrm{div}{(F)}  + \frac{m}{2}K\\
  &= \mathrm{div}{(f)}-\frac{m}{2}\sum_{\mathfrak a\in \mathfrak R_\Gamma, \ \ elliptic} (1-1/e_{\mathfrak a}) \mathfrak a-\frac{m}{2}
  \sum_{\substack{\mathfrak b \in \mathfrak R_\Gamma, \\ cusp}} \mathfrak b
\end{align*}
Next, using \cite[Lemma 2.2]{MuicKodrnja} (see (\ref{mhd-60000000000000000}) in this paper), the right--hand side becomes
$$
\mathfrak c'_f - \sum_{\mathfrak a\in \mathfrak
R_\Gamma, \ \ elliptic}  \left[\frac{m}{2}(1-1/e_{\mathfrak  a})\right]\mathfrak a - \frac{m}{2}
  \sum_{\substack{\mathfrak b \in \mathfrak R_\Gamma, \\ cusp}} \mathfrak b \ge 0
$$
  by the definition of $S_m^H(\Gamma)$. Hence,  $f/F^{\frac{m}{2}}\in L\left(\frac{m}{2}K\right)$. Now, comparing the dimensions  in (\ref{cts-2}), we obtain their equality.
This proves (\ref{cts-2}).

Next, let $W$ be any finite dimensional $\mathbb C$--vector space. Let $\text{Symm}^k(W)$ be symmetric tensors of degree $k\ge 1$.
Then, by Max Noether theorem (\cite{Miranda}, Chapter VII, Corollary 3.27) the multiplication induces surjective map
$\text{Symm}^k{\left(L\left(K\right)\right)}\twoheadrightarrow L\left(\frac{m}{2}K\right)$.
The lemma follows.
\end{proof}

 \vskip .2in
 Now, after all of these preparations, we come to the main result of the paper.  It gives a good criterion 
  for {\bf testing} that  $\mathfrak a_\infty$ is  a  $\frac{m}{2}$--Weierstrass point for $m\ge 4$.
 We give examples of explicit computations below.

 \begin{Thm}\label{cts-500000}  Let $m\ge 4$ be an even integer. Assume that  $\mathfrak R_\Gamma$ is not hyperelliptic. Assume that $\mathfrak a_\infty$ is a cusp for $\Gamma$.
   Let us select a basis $f_0, \ldots, f_{g-1}$, $g=g(\Gamma)$, of $S_2(\Gamma)$ as listed by their $q$--expansions using SAGE system. Compute $q$--expansions of all monomials
   $$
   f_0^{\alpha_0}f_1^{\alpha_1}\cdots f_{g-1}^{\alpha_{g-1}}, \ \  \alpha_i\in \mathbb Z_{\ge 0}, \ \sum_{i=0}^{g-1} \alpha_i=\frac{m}{2}.
   $$
   Then,  $\mathfrak a_\infty$ is {\bf not} a   $\frac{m}{2}$--Weierstrass point if and only if there exist  $\mathbb C$--linear combinations of such monomials,  say 
   $F_1, \ldots F_t$, $t= (m-1)(g-1)$, such that their $q$--expansions are of the form
      $$
      F_u=a_uq^{u+m/2-1}+ \text{higher order terms in $q$}, \ \ 1\le u\le t,
      $$
      where 
      $$
      a_u\in \mathbb C, \ \  a_u\neq 0.
      $$
\end{Thm}
\begin{proof}  We combine Lemmas \ref{cts-50000} and \ref{mhd-9} (x).
  \end{proof}

We make the method of  Theorem \ref{cts-500000} more explicit as follows for $\Gamma=\Gamma_0(N)$. First, 
the number of monomials is
   $$
   \binom{g+m/2-1}{m/2}.
   $$
Then, by selecting the first $m/2+m \cdot (g-1)$ terms from $q$--expansions of the monomials, we can create the matrix of size  
   $$
   \binom{g+m/2-1}{m/2} \times \left(\frac{m}{2}+m \cdot (g-1)\right).
   $$ 
Then, we perform suitable integral Gaussian elimination method to transform the matrix into row echelon form.
The procedure is as follows. We successively sort and transform the row matrices to cancel the leading row coefficients with the same number of leading zeros as their predecessor.
We use the {\it Quicksort algorithm} for sorting. We obtain  the transformed matrix and the transformation matrix.
The non-null rows of the transformed matrix give the $q$-expansions of the basis elements of $S_m^H(\Gamma)$, and the corresponding rows of the transformation matrix give the corresponding
linear combinations of monomials. Using this method, we check the following from Theorem \ref{cts-500000}.

Let $m=4$. Then, for $X_0(34)$  the basis of   $S_4^H(\Gamma_0(34))$ is given by 
\begin{align*}
f_0^2 & =q^{2}-4q^{5}-4q^{6}+12q^{8}+12q^{9}-2q^{10} \\
f_0f_1& = q^{3}-q^{5}-2q^{6}-2q^{7}+2q^{8}+5q^{9}+2q^{10} \\
f_0f_2 & = q^{4}-2q^{5}-q^{6}-q^{7}+6q^{8}+6q^{9}+2q^{10} \\
 -f_1^2 + f_0f_2& = -2q^{5}+q^{6}-q^{7}+5q^{8}+6q^{9}+4q^{10} \\
-f_1^2 + f_0f_2 + 2f_1f_2& =-3q^{6}-5q^{7}+11q^{8}+16q^{9}+2q^{10} \\
-f_1^2 + f_0f_2 + 2f_1f_2 + 3f_2^2 &= -17q^{7}+17q^{8}+34q^{9}+17q^{10} \\
\end{align*}
Their first exponents  are $\frac{m}{2}=2, 3, 4, 5, 6, \frac{m}{2}+ (m-1)(g-1)-1=7$ which shows that
$\mathfrak a_\infty$ is not $2$--Weierstrass point for $X_0(34)$.

For $X_0(55)$,  the basis of   $S_4^H(\Gamma_0(55))$ is given by
\begin{align*}
 f_0^2 & = q^{2}-2q^{8}+\cdots  \\
 f_0f_1& = q^{3}-2q^{7}+ \cdots  \\
 f_0f_2&= q^{4}-2q^{7}+ \cdots\\
  f_0f_3& =q^{5}-2q^{7} + \cdots\\
 f_0f_4 &=q^{6}-2q^{11} + \cdots\\
 -f_1f_2 + f_0f_3 &=-2q^{7}+q^{8} + \cdots\\
 -f_1f_2 + f_0f_3 + 2f_2f_3  &=q^{8}+2q^{9} + \cdots\\
  -f_1f_2 + f_0f_3 + 2f_2f_3 - f_3^2 &=2q^{9}-q^{10} + \cdots\\
  -f_1f_2 + f_0f_3 + 2f_2f_3 - f_3^2 - 2f_3f_4 & =-q^{10}+11q^{12} + \cdots\\
 -f_1f_2 + f_0f_3 + 2f_2f_3 - f_3^2 - 2f_3f_4 + f_4^2 & =11q^{12}-11q^{13} + \cdots\\
 -f_1f_2 - f_2^2 + f_0f_3 + 2f_2f_3 - f_3^2 + f_0f_4 - 6f_3f_4 - f_4^2& =-22q^{13}+44q^{15}+ \cdots\\
 -f_2^2 + f_3^2 + f_0f_4 - f_2f_4 - 4f_3f_4 + 2f_4^2 &=-22q^{14}+22q^{15}+ \cdots\\
\end{align*}
The last exponent is $14>  \frac{m}{2}+ (m-1)(g-1)-1=13$. So,  $\mathfrak a_\infty$ is a
$2$--Weierstrass point for $X_0(55)$.

   \vskip .2in
We end the section with the following remark. When $\mathfrak R_\Gamma$ is hyperelliptic,  $S_{m, 2}^H(\Gamma)$ could be   a  proper subspace of $S_{m}^H(\Gamma)$.
For example,  assume that $g(\Gamma)= 2$. Let $f_0, f_1$ be a basis of $S_2(\Gamma)$. Then, for any even integer $m\ge 4$,
$f_0^uf_1^{\frac{m}{2}-u}$, $0\le u \le m$ is a basis of $S_{m, 2}^H(\Gamma)$. Therefore,
  $$
  \dim{S_{m}^H(\Gamma)}= (m-1)(g(\Gamma)-1)= m-1\ge  \frac{m}{2}+1=  \dim{S_{m, 2}^H(\Gamma)}, \ \ \text{for} \ m\ge 4.
  $$
 Thus,  $S_{4, 2}^H(\Gamma)= S_{4}^H(\Gamma)$ while $S_{m, 2}^H(\Gamma)\subsetneq S_{m}^H(\Gamma)$ for $m\ge 6$.

\end{document}